\documentclass[12pt,letterpaper]{article}
\usepackage{amsmath,amsthm,amssymb,amsfonts,color}

\newfont{\lie}{eufm10 at 11pt}
\newfont{\liepequenos}{eufm10 at 10pt}

\newfont{\corpos}{msbm10 at 11pt}
\newfont{\corpospequenos}{msbm10 at 10pt}

\newcommand{\C}{\mbox{\corpos \symbol{67}}}          
\newcommand{\Octoni}{\mbox{\corpos \symbol{79}}}     
\newcommand{\Proj}{\mbox{\corpos \symbol{80}}}        
\newcommand{\R}{\mbox{\corpos \symbol{82}}}          
\newcommand{\Z}{\mbox{\corpos \symbol{90}}}          

\newcommand{\g}{\mbox{\lie g}}       %

\newcommand{\sol}{\mbox{\lie so}}

\newcommand{\rr}{\rightarrow}
\newcommand{\lrr}{\longrightarrow}

\newcommand{\rn}{\R^{n}}                    %

\newcommand{\hnab}{{\cal H}^{D}}             %
\newcommand{\calR}{{\cal R}}             %

\newcommand{\na}{{\nabla}}
\newcommand{\nag}{{\nabla^g}}
\newcommand{\Tr}[1]{{\mathrm{Tr}}\,{#1}}

\newcommand{\End}[1]{{\mathrm{End}}\,{#1}}

\newcommand{\Id}{{\mathrm{1}}}

\newcommand{\dx}{{\mathrm{d}}}
\newcommand{\bx}{{\circlearrowright}}
\newcommand{\inv}[1]{{#1}^{-1}}
\newcommand{\papa}[2]{\frac{\partial#1}{\partial#2}}

\newcommand{\dastC}{{\dx*\phi\,_{\mathrm{curv}}}}
\newcommand{\dastT}{{\dx*\phi\,_{\mathrm{tors}}}}
\newcommand{\astil}{{\tilde{*}}}

\newcommand{\vol}{{\mathrm{vol}}}
\newcommand{\Vol}{{\mathrm{Vol}}}
\newcommand{\ric}{{\mathrm{Ric}\,}}

\newcommand{\Dstar}{D^\star}
\newcommand{\gwistor}{{\em{gwistor}}}

\newtheorem{teo}{Theorem}[section]

\newtheorem{coro}{Corollary}[section]
\newtheorem{prop}{Proposition}[section]

\newenvironment{Defi}[1][Definition]{\begin{trivlist}
\item[\hskip \labelsep {\bfseries #1}]}{\end{trivlist}}

\newenvironment{Rema}[1][Remark]{\begin{trivlist}
\item[\hskip \labelsep {\bfseries #1}]}{\end{trivlist}}

\newenvironment{meuenumerate}
{\begin{enumerate}
  \setlength{\itemsep}{1pt}
  \setlength{\parskip}{0pt}
  \setlength{\parsep}{0pt}}
{\end{enumerate}}

\def\cyclic{\mathop{\kern0.9ex{{+}
\kern-2.2ex\raise-.28ex\hbox{\Large\hbox{$\circlearrowright$}}}}\limits}

\title{On the $G_2$ bundle of a Riemannian 4-manifold}

\author{R. Albuquerque\footnote{Departamento de Matem\'atica da Universidade de \'Evora and Centro de Investiga\c c\~ao em Matem\'atica e Aplica\c c\~oes (CIMA-U\'E), Rua Rom\~ao Ramalho, 59, 671-7000 \'Evora, Portugal.}\\ rpa@dmat.uevora.pt}

\begin{document}


\maketitle


\begin{abstract}

We study the natural $G_2$ structure on the unit tangent sphere bundle $SM$ of any given orientable Riemannian 4-manifold $M$, as it was discovered in \cite{AlbSal}. A name is proposed for the space. We work in the context of metric connections, or so called geometry with torsion, and describe the components of the torsion of the connection which imply certain equations of the $G_2$ structure. This article is devoted to finding the $G_2$-torsion tensors which classify our structure according to the theory in \cite{FerGray}.

\end{abstract}

\tableofcontents

\vspace*{10mm}


\vspace*{4mm}


\vspace*{10mm}

The author acknowledges the support of Funda\c{c}\~{a}o Ci\^{e}ncia e Tecnologia, Portugal, either through CIMA-U\'E, Centro de Investiga\c c\~ao em Matem\'atica e Aplica\c c\~oes da Universidade de \'Evora and through SFRH/BSAS/895/2009.

\markright{\sl\hfill  R. Albuquerque \hfill}

\vspace*{10mm}

\section{Recalling the theory}
\label{Rott}

\subsection{Introduction}

A $G_2$ structure on a 7-dimensional Riemannian manifold $(N,g)$ is given by a reduction from the principal orthonormal frame $O(7)$-bundle to a principal $G_2$-subbundle. The Lie group $G_2$ coincides with $\mathrm{Aut}(\Octoni)$ and so is contained in $SO(7)$. The structure is characterized by a 3-form $\phi$ on $N$ which, for any vector fields $X,Y,Z$, reads $\phi(X,Y,Z)=\langle XY,Z\rangle$ with $XY$ denoting the octonionic product. The manifold $N$ admits a non-degenerate $\phi$ if the first two Stiefel-Whitney classes vanish.

The geometry of $G_2$ structures has been having much attention. Ahead of the topological determination, the Riemannian holonomy of a pair $(N,\phi)$ is included in $G_2$ if, and only if, $\phi$ is parallel. It was proved by A. Gray the latter differential condition corresponds with $\phi$ being harmonic. But it is a particularly difficult problem to find compact examples satisfying such equations, corresponding with one of the possible cases shown in the celebrated list of irreducible holonomies, as one may deduce from \cite{Bryant1,BrySal,Joy}.

It is then natural to impose conditions on $G_2$ structures on Riemannian 7-manifolds which solve less stringent equations. We refer to those following from the decomposition into irreducible components of $\Lambda^*T^*N$ as a $G_2$-module, which yield particularly interesting notions in degrees 4 and 5, where $\dx\phi$ and $\dx*\phi$ belong, with equivalents in degrees 2 and 3 by Hodge-$*$ duality. They were first discovered in \cite{FerGray}. Nowadays there is manifest interest in co-calibrated structures, defined by the condition $\delta\phi=0$, as we see from \cite{Agri,FriIva1,FriIva2} and some literature in Physics. In particular, the case of nearly parallel structures, $\dx\phi=c*\phi$ with $c$ constant, cf. \cite{FriKaMoSe}.

In this work we resume to a particular construction on the unit tangent sphere bundle,
\begin{equation}\label{spherebundle0} 
SM=\bigl\{u\in TM:\ \|u\|=1\bigr\},
\end{equation}
of an oriented Riemannian 4-manifold $M$. This was founded in \cite{AlbSal} and consists of smoothly assigning to each $T_uSM,\ u\in SM$, the structure of the imaginary part of an octonionic structure on $T_uTM$, having $\R u$ as the reals line. Merely induced by the orientation and metric on $M$, we coin this natural $G_2$-structure associated to the base $M$ with the name of ``gwistor space'' of $M$. In particular, every $SM$ is a spin manifold.

So far, it has been deduced how the respective 3-form $\phi$ is written, computed its exterior differential and co-differential and shown some consequences and examples. A \gwistor\ space $G_2$-structure cannot be harmonic. However, it is co-calibrated if, and only if, $M$ is an Einstein manifold, which leads to a plenty of new examples. This theorem was first proved in \cite{AlbSal}, up to a slight mistake which has now been cleared and duely submitted to the respective journal. Here we prove the result again. Before, the construction used the Levi-Civita connection of $M$ to produce a splitting of $TSM$, a canonical splitting into horizontal and vertical subspaces. Clearly, this is also possible if we use any other metric connection, so we develop here in this direction. We determine the Levi-Civita connection of $SM$ and compute the relevant derivatives of the $G_2$ structure. Thus one of our purposes is to analyse the variations of the $G_2$ structure on the unit tangent sphere bundle depending on the torsion of the connection on $M$. The subgroups $SO(3)$ and $SO(4)$ play some intricate role, yet to be truly understood. Our main objective here is to deduce the irreducible components of $\dx\phi$ and $\delta\phi$ under the exceptional Lie group $G_2$, also called torsion tensors, and hence be able to classify our \gwistor\ space. We ask the reader to carefully distinguish between the two meanings of the word `torsion' depending on the context. The difficulties found for the expression of those tensors, theorem \ref{astorsoesenfim}, explain to some extent the ever-hidden structure of $G_2$ next to 4-manifolds.

Many problems and questions may be raised about \gwistor\ space, as the reader may deduce for himself. We notice  the interesting feature of $SM$ which follows by changing the sign of the metric only on the direction of the $S^3$ fibres. This corresponds with the split octonions and the non-compact dual of $G_2$. Hence we may construct a $\widetilde{G}_2$ structure using the same set of ideas.

From the beginning, our work has been much influenced from a good understanding of twistor theory according to \cite{Obri}. We express our thanks to Ilka Agricola, Thomas Friedrich and Ivan Minchev for fruitful conversations. As this work was partly done during the author's sabbatical leave at Philipps Universit\"at, Marburg, we kindly acknowledge the hospitality and friendly mathematical environment.

The author dedicates this work to the friend and great man Carlos Grilo.

\subsection{The unit tangent sphere bundle, its metric and Levi-Civita connection}
\label{TsbimaLCc}

Let $(M,g)$ be a smooth orientable Riemannian manifold of dimension $n$. To represent $g$ we also  use $\langle\:,\:\rangle$. Let $D$ be a metric connection on $M$, ie. a linear connection for which the metric is parallel. We denote
\begin{equation}\label{spherebundle} 
SM=\bigl\{u\in TM\,:\ \|u\|=1\bigr\},
\end{equation}
the total space of the unit tangent sphere bundle, and we let $f$ denote the projection onto the base $M$. Notice $SM$ is always orientable, regardless of $M$. The connection $D$ induces a direct sum decomposition $T\,SM={\cal V}\oplus\hnab$ into vertical and horizontal distributions:  ${\cal V}=\ker\dx f$ and $\hnab\stackrel{\dx f}\simeq f^*TM$. Furthermore, the bundle $\cal V$ may be identified with the vector subbundle of $\inv{f}TM$ (occasionally we use the notation $\inv{f}$ to refer to the vertical side) such that ${\cal V}_u=u^\perp\subset T_{f(u)}M$. In the following we denote by $U$ the `vertical section' $U:SM\rr \inv{f}TM$ defined by $U_u=u$. The  \textit{ubiquitous} character of $U$ is even repeated since the vector appears also on the horizontal subspace.

Recall that $\hnab=\{X\in T\,SM|\ (f^*D)_XU=0\}$ and $f^*D_{\,Y}U=Y$ for any $Y\in{\cal V}$, a result whose proof the reader may see in \cite{AlbSal}. In this regard, it is worthwhile to think of $U$ as if it would vary only on vertical directions, and to notice that, for two given sections $X,Z\in\Omega^0(TM)$, ie. two vector fields on $M$, the vertical part of $\dx X(Z)$ is precisely $D_ZX$. Indeed, in this case $X^*U_x=U_{X(x)}=X_x$, for $x\in M$, hence
\[ (\dx X(Z))^v=f^*D_{\dx X(Z)}U=(X^*f^*D)_ZX^*U=D_ZX.\]
We use $\,\cdot\,^v$ and $\,\cdot\,^h$ to denote the obvious projections.

Now we endow $SM$ with the unique Riemannian structure for which the two distributions $\hnab$ and $\cal V$ sit orthogonally inside $T\,SM$ and the identifications with $f^*TM$ are isometric. This metric is the induced metric on $SM$ from the well known Sasaki metric on $TM$. We still denote it by $g=\langle\:,\:\rangle$.

The tangent bundle $TTM$ has a natural metric connection: $D^*=f^*D\oplus\inv{f}D$; it preserves the splitting into the horizontal and vertical, hence $D^*$-parallel, tangent subbundles. Both projections $\,\cdot\,^v$ and $\,\cdot\,^h$ are parallel.

$T\,SM$ inherits a metric connection too via the pull-back connection $D^*$, still preserving the splitting and which we shall denote by $\Dstar$. Notice the star, depicting the fact that on vertical tangent directions we must add a correction term
\begin{eqnarray}\label{correcterm}
\Dstar_Y X^v &=& D^*_Y X^v-\langle D^*_Y X^v,U\rangle U\ =\  D^*_Y X^v+\langle X^v,Y^v\rangle U.
\end{eqnarray}
Clearly $-\langle D^*_Y X,U\rangle =\langle X,D^*_YU\rangle=\langle X^v,Y^v\rangle$ for $X,Y\in T\,SM$. This stems from the geometry of the sphere with the round metric, in quest of a metric linear connection.

We shall use explicitly the \textit{mirror} isomorphism $\theta:\hnab\rr\inv{f}TM\supset{\cal V}$
defined by $D^*$-parallel isometric identification of the pull-back bundle of $TM$. We may extend $\theta$ by 0 to the vertical tangent bundle, so in fact we have
\begin{equation}\label{teta}
\theta:T\,TM\lrr T\,TM,\qquad\ \ \theta^t\theta(X^h)=X^h,\qquad\theta\theta^t(X^v)=X^v.
\end{equation}

Let $R^*=f^*R^{D}=R^{f^*D}$ denote the curvature of $f^*D$. The latter identity follows by tensoriality. We use $R^D(X,Y)=D_XD_Y-D_YD_X-D_{[X,Y]}$ and view $R^*U$ as the following $\cal V$-valued 2-form $\calR$ over $SM$:
\begin{equation}\label{definicaocalR}
 \calR_u(X,Y)=R^*(X,Y)U_u=\sum_{i=1}^{n-1}\langle R^D(f_*X,f_*Y)u,e_i\rangle e_{i+3},\qquad\forall u\in SM,
\end{equation}
where $e_0,\ldots,e_{n-1}$ is an orthonormal frame of $\hnab$ such that $e_0=\theta^tU$ and $e_{i+n-1}=\theta e_i,\ i\geq1$. Notice $R^*U\perp U$ because $D$ is an $\sol(n)$-connection.

Let $T^*$ denote the $\hnab$-valued form $f^*T^D$, the pull-back of the torsion tensor of $D$: $T^D(X,Y)=D_XY-D_YX-[X,Y]$. Like $R^*$, $T^*$ varies only on two horizontal directions.
\begin{prop}
 $T^{\Dstar}=T^*+\calR$.
\end{prop}
\begin{proof}
 Only the vertical part here is unexpected:
\begin{eqnarray*} 
 (T^{\Dstar}(X,Y))^v &=& \Dstar_XY^v-\Dstar_YX^v-[X,Y]^v\ =\  D^*_XD^*_YU+\langle X^v,Y^v\rangle U \\ & &   \quad\ \ \ \ -D^*_YD^*_XU-\langle X^v,Y^v\rangle U-D^*_{[X,Y]}U\ =\ \calR(X,Y).
\end{eqnarray*}
The horizontal part occurs by definition.    
\end{proof}
\begin{teo}
The Levi-Civita connection $\nag$ of $SM$ is given by
\begin{equation}\label{lcTM}
\nag_XY=\Dstar_XY-\dfrac{1}{2}\calR(X,Y)+A(X,Y)+\tau(X,Y)
\end{equation}
where $A,\tau$ are $\hnab$-valued tensors defined by
\begin{equation}\label{AlcTM}
\langle A(X,Y),Z\rangle=\dfrac{1}{2}\langle \calR(X,Z), Y\rangle + \dfrac{1}{2}\langle \calR(Y,Z), X\rangle
\end{equation}
and
\begin{equation}\label{taulcTM}
\tau(X,Y,Z)=\langle \tau(X,Y),Z^h\rangle= \dfrac{1}{2}\bigl(T^*(Y,X,Z)+T^*(X,Z,Y)+T^*(Y,Z,X)\bigr),
\end{equation}
with $T^*(X,Y,Z)=\langle T^*(X,Y),Z\rangle$, for any vector fields $X,Y,Z$ over $SM$.
\end{teo}
\begin{proof}
This is a straightforward verification of identities $\dx\langle X,Y\rangle=\langle \nag X,Y\rangle+\langle X,\nag Y\rangle$ and $T^\nag=0$. This checking was done in \cite{AlbSal} for the torsion free case; the present one being equally trivial. Notice 
\[ \tau(X,Y)-\tau(Y,X)=-T^*(X,Y),  \]
as we wish, and the 3-tensor $\tau$ is skew in $Y,Z$. 
\end{proof}
In sum, $D^*$ and $\Dstar$ preserve types, $\calR$ takes vertical values and $A$ and $\tau$ take horizontal values. Moreover, $A(X,Y)=0$ if $X,Y$ are both horizontal or both vertical, whereas $\calR$ and $\tau$ vanish if one direction is vertical. The $A$ tensor may be written as
\begin{equation}\label{AlcTMhvsv}
\langle A(X,Y),Z^h\rangle=\dfrac{1}{2}\langle \calR(X^h,Z^h), Y^v\rangle + \dfrac{1}{2}\langle \calR(Y^h,Z^h), X^v\rangle .
\end{equation}
Notice the well known result following trivially from the formula above: $\hnab$ corresponds with an integrable distribution if, and only if, $D$ is flat. Also, by the skew-symmetries in $X,Y$, $\tau=0$ if, and only if, $T^D=0$.

\section{The natural $G_2$ structure on $SM$}
\label{TcG2s}

\subsection{\textit{Gwistor} space}

Any quaternionic-Hermitian structure on an oriented Euclidean 4-vector space $Q$ arises from a choice of a unit direction $u\in S^3\subset Q$, from the metric and the orientation. Indeed, $u$ plays the role of a unit and a generator of the real line in the quaternions; the volume form $\vol_Q$ coupled with $u$ determines a cross product in $u^\perp$. We use 
\begin{equation}
 \langle X\times Y,Z\rangle=\vol(u,X,Y,Z),\qquad\ \ \forall X,Y,Z\in u^\perp
\end{equation}
and this is enough to have a quaternionic-Hermitian line:
\begin{equation}
 (\lambda_1u+X)(\lambda_2u+Y)\ =\ (\lambda_1\lambda_2-\langle X,Y\rangle)u+\lambda_2X+\lambda_1Y+X\times Y
\end{equation}
with conjugation map $\overline{\lambda u+X}=\lambda u-X$. Then an octonionic structure is built on $Q\oplus Q$ by the well known Cayley-Dickson process and thence a Lie group $\mathrm{Aut}(Q\oplus Q)$ isomorphic to $G_2$ is well defined. For more details see \cite{AlbSal,HarLaw}.

Now let $(M,g)$ be an oriented Riemannian 4-manifold. Consider again the differential geometric setting of the unit tangent sphere $SM$ of last section. We define an octonionic structure on $T_uTM\simeq T_xM\oplus T_xM$, for any $u\in S_xM$ and $x\in M$, by the process above. The imaginary part of such non-associative division algebra is precisely $T_uSM$. And this is the natural $G_2$-structure of an oriented Riemannian 4-manifold we have been referring to. Because of the tautological construction and its resemblance with twistor theory we propose the following.
\begin{Defi}
We call \gwistor\ {\em{bundle of $M$ associated to $D$}} or simply \gwistor\ \textit{space of $M$} to the unit tangent sphere bundle $SM$ of an oriented Riemannian 4-manifold $M$ together with its natural $G_2$ structure induced from a linear metric connection $D$.
\end{Defi}
It is possible to define a line of split-octonions in the same way as above, ie. with a pseudo metric of signature $(4,-4)$ on $Q\oplus Q$ and the same algebraic process. The automorphisms group is $\tilde{G}_2$, the non-compact dual of the semisimple Lie group $G_2$. This translates into pseudo-Riemannian $SM$ over the oriented Riemannian $M$, providing a new way to construct $\tilde{G}_2$ structures. The structure equations and the integrability conditions are still unknown. If one compares, already twistor theory allows analogous variation.

We also have the following important remark, to which we are led on the course of emergence of \gwistor\ space. Recall the endomorphism $\theta:TTM\rr TTM$ defined in (\ref{teta}) and let $\xi=2\theta^tU$ and $\eta=\frac{1}{2}\mu=\frac{1}{2}(\theta^tU)^\flat$. Also let $\varphi=\theta-U\mu-\theta^t$.
\begin{teo}[Y. Tashiro]
 $(SM,\frac{1}{4}g,\xi,\eta,\varphi)$ is an almost contact structure.
\end{teo}
The theorem, or rather the structure, is already known and valid in any dimension, cf. \cite{Blair,Tash}. Further properties may be deduced within the lines of this article.

\subsection{The structure equations}
\label{tse}

We first introduce a tool which will be frequently used. The proofs of some assertions are canonical or straightforward.
\begin{Rema}
Given any $p$-tensor $\eta$ and any endomorphisms $B_i$ of the tangent bundle we let $\eta\circ (B_1\wedge\ldots\wedge B_p)$ denote the new $p$-tensor defined by
\begin{equation}\label{esttensorcont}
\eta\circ (B_1\wedge\ldots\wedge B_p)(Y_1,\ldots,Y_p)=\sum_{\sigma\in S_p}\mathrm{sg}(\sigma)\eta(B_1Y_{\sigma_1},\ldots,B_pY_{\sigma_p}).
\end{equation}
This contraction obeys a simple Leibniz rule under covariant differentiation, with no minus signs attached as expected. In particular, if $\eta$ is a $p$-form, then $\eta\circ\wedge^p\Id=p!\,\eta$. Furthermore one verifies that for a wedge of 1-forms, $\eta_1\wedge\ldots\wedge\eta_p\circ(B_1\wedge\ldots\wedge B_p)= \sum_{\sigma\in S_p} \eta_1\circ B_{\sigma_1}\wedge\ldots\wedge\eta_p\circ B_{\sigma_p}$. Notice that for a 2-form $\eta$ and an endomorphism $B$, clearly \,$\eta\circ B\wedge B\,(X,Y)=2\eta(BX,BY)$.
\end{Rema}
Now any $G_2$ structure on a 7 dimensional Riemannian manifold is entirely determined by a non-degenerate 3-form, say $\phi$. We proceed to find our 3-form. As in \cite{AlbSal}, we define a 3-form $\alpha=U\lrcorner\,\inv{f}\vol_M\in\Omega^0_{SM}(\Lambda^3{\cal V}^*)\subset\Omega^3$, a 1-form $\mu$ by $\mu(X)=\langle U,\theta X\rangle$ and a 2-form $\beta$ by $g\circ \theta\wedge\Id$. Using the natural contraction, we define $\alpha_1,\alpha_2\in\Omega^3$ by
\begin{equation}\label{alfas} 
\alpha_1=\frac{1}{2}\alpha\circ(\theta\wedge\Id\wedge\Id), \hspace{2cm}
\alpha_2=\frac{1}{2}\alpha\circ(\theta\wedge\theta\wedge\Id)
\end{equation} 
where $\Id$ is the identity map. To understand this, notice $\alpha$ is a 3-form and hence $\alpha_1=\cyclic_{X,Y,Z}\alpha(\theta X,Y,Z)$, the cyclic sum. And $\beta(X,Y)=\langle\theta X,Y\rangle -\langle\theta Y,X\rangle=\langle\theta X^h,Y^v\rangle-\langle\theta Y^h,X^v\rangle$. Finally the associated 3-form $\phi$ of the \gwistor\ bundle of $M$ and its image under Hodge-$*$ are given respectively by
\begin{equation}\label{phieestrelaphi}
\phi=\alpha+\mu\wedge\beta-\alpha_2\qquad\mbox{and}\qquad *\phi=\vol^*_M-\frac{1}{2}\beta^2-\mu\wedge\alpha_1,
\end{equation}
where we denote $\vol_M^*=f^*\vol_M$. The following result is quite useful.
\begin{prop}[first structure equations, \cite{AlbSal}] \label{bse}
We have the following relations:
\begin{equation}\label{bse1} 
\begin{split}
*\alpha=\vol^*_M,\qquad\ \ *\alpha_1=-\mu\wedge\alpha_2,\qquad\ \ *\alpha_2=\mu\wedge\alpha_1,\hspace{1.7cm} \\
*\beta=-\frac{1}{2}\mu\wedge\beta^2,\qquad\ \  *\beta^2=-2\mu\wedge\beta,\qquad\ \ \beta^3\wedge\mu=-6\Vol_{SM}.
\end{split}
\end{equation}
Writing $\alpha_0=\alpha$ we have also
\begin{equation}\label{bse2} 
\begin{split}
\alpha_1\wedge\alpha_2=3*\mu=-\frac{1}{2}\beta^3,\qquad\ \  \beta\wedge\alpha_i=\beta\wedge*\alpha_i=\alpha_0\wedge\alpha_i=0,\ \ \forall i=0,1,2, \\
\alpha\wedge\phi=\alpha_2\wedge\phi =*\alpha_1\wedge\phi=0\qquad\qquad *\alpha\wedge\phi=\alpha\wedge*\phi=\Vol_{SM}.\ \ 
\end{split}
\end{equation}
\end{prop}
To study these forms we use a frame. Let $e_0=\theta^tU,e_1,\ldots,e_6$ denote the direct orthonormal basis of $TSM$ induced from $e_0,\ldots,e_3$, a direct orthonormal basis of $\hnab$, as in (\ref{definicaocalR}). It is easy to prove the existence of such frames: we fix a direct o.n. frame $f_1,\ldots,f_4$ on an open set in $M$ and then take Cartesian coordinates $(u_1,\ldots,u_4)\in S^3$ to write $e_0=u=\sum u_if_i$. Then $e_1,e_2,e_3$ follow by a well known transformation in $Sp(1)$ of the $u_i$, cf. section \ref{ttgs}. The mirror map $\theta$ gives $e_4,e_5,e_6$. Now, by definition, $\alpha=e^{456}$ and $\mu=e^0$. Thus $*\alpha=e^{0123}=\vol^*_M$. It is also trivial to see $\beta=e^{14}+e^{25}+e^{36}$. From direct inspection on $\alpha$ composed with $\theta$ we find
\begin{equation}\label{alphascoef}
  \alpha_1=e^{156}+e^{264}+e^{345}\qquad\ \ \ \mbox{and}\qquad\ \ \ \alpha_2=e^{126}+e^{234}+e^{315} .
\end{equation}
Hence $*\alpha_1=-e^{0234}+e^{0135}-e^{0126}=-\mu\wedge\alpha_2$ and $*\alpha_2=e^{0345}+e^{0156}+e^{0264}=\mu\wedge\alpha_1$. Now $\beta^3=(2e^{1425}+2e^{1436}+2e^{2536})\wedge\beta=6e^{142536}=-6*\mu$.
Finally, $*\beta=e^{02356}+e^{01346}+e^{01245}=-\frac{1}{2}\mu\wedge\beta^2$ and  $*\beta^2=-2*(e^{1245}+e^{1346}+e^{2356})=-2(e^{036}+e^{025}+e^{014})=-2\mu\wedge\beta$. The other relations in (\ref{bse1}) and (\ref{bse2}) follow as easily. Finally
\begin{equation}
 |\phi|^2=*(\phi\wedge*\phi)=*(1+\frac{6}{2}+\frac{6}{2})\Vol_{SM}=7 .
\end{equation}
We shall see an example after a final remark on the structure.
\begin{Rema}
There are plenty of tensors in \gwistor\ space which account for the richness of the unit tangent sphere bundle, but a few are not needed to describe $\phi,*\phi$ or its derivatives. Nevertheless, for the sake of completeness, we will do the systematic study and computation of their derivatives. The first is the tensor $\theta^tU=\mu^\sharp$, a canonical vector field, and the second is
\begin{equation}\label{alpha3}
 \alpha_3=\frac{1}{6}\alpha\circ(\theta\wedge\theta\wedge\theta)=\theta^tU\lrcorner f^*\vol_M.
\end{equation}
In the earlier frame, we get $\alpha_3=e^{123}$. Also notice $TSM$ is isomorphic to the direct sum of two vector bundles of rank 3 and a line bundle.
\end{Rema}

\subsection{The trivial \gwistor\ space}
\label{ttgs}

Before going further it is interesting to observe the case of $\R^4$ with canonical metric and trivial connection $\dx$. Here $SM=\R^4\times S^3\subset\R^8$, in which we use coordinates $(x,u)=(x_1,\ldots,x_4,u_1,\ldots,u_4)$. Of course the $\papa{ }{x_i}$ are orthonormal and horizontal, hence the previously announced co-frame $e^0,e^1,\ldots,e^6$ may be given by the identities $e^0=\mu=\sum u_i\dx x_i$ and
\begin{equation}\label{coframe}
 \left[\begin{array}{cc}e^1&e^4\\e^2&e^5\\e^3&e^6\end{array}\right]
=\left[\begin{array}{cccc}-u_2& u_1 & -u_4 & u_3\\ 
-u_3& u_4 & u_1 & -u_2\\ -u_4& -u_3 & u_2 & u_1\end{array}\right]
\left[\begin{array}{cc}\dx x_1&\dx u_1\\ \dx x_2&\dx u_2\\ \dx x_3&\dx u_3\\ \dx x_4&\dx u_4\end{array}\right].
\end{equation}
A simple but quite long computation yields:
\begin{eqnarray*}
 \alpha&=&u_1\dx u_{234}-u_2\dx u_{134}+u_3\dx u_{124}-u_4\dx u_{123}\\
\mu\wedge\beta&=& u_1(\xi_{12,2}+\xi_{13,3}+\xi_{14,4})+u_2(\xi_{21,1}+\xi_{23,3}+\xi_{24,4})\\
& &\ \ \ +u_3(\xi_{31,1}+\xi_{32,2}+\xi_{34,4})+u_4(\xi_{41,1}+\xi_{42,2}+\xi_{43,3})\\
\alpha_2&=&u_1(\xi_{23,4}-\xi_{24,3}+\xi_{34,2})-u_2(\xi_{13,4}-\xi_{14,3}+\xi_{34,1})\\
& &\ \ \ +u_3(\xi_{12,4}+\xi_{24,1}-\xi_{14,2})-u_4(\xi_{12,3}-\xi_{13,2}+\xi_{23,1})\\
\alpha_1&=&u_1(\xi_{2,34}-\xi_{3,24}+\xi_{4,23})-u_2(\xi_{1,34}-\xi_{3,14}+\xi_{4,13})\\
& &\ \ \ +u_3(\xi_{1,24}-\xi_{2,14}+\xi_{4,12})-u_4(\xi_{1,23}-\xi_{2,13}+\xi_{3,12})
\end{eqnarray*}
where $\dx u_{ijk}=\dx u_i\wedge\dx u_j\wedge\dx u_k$ and $\xi_{ij,k}=\dx x_i\wedge\dx x_j\wedge \dx u_k$ and $\xi_{i,jk}=\dx x_i\wedge\dx u_j\wedge \dx u_k$. Then we find
\begin{equation}\label{flatcase}
 \dx\alpha=0,\qquad \dx\mu=-\beta,\qquad \dx\alpha_2=2\mu\alpha_1,\qquad \dx\alpha_1=3\mu\alpha. 
\end{equation}
Notice these results are valid subject to the condition $u_1^2+\cdots+u_4^2=1$. Their purpose is to explicitly describe the 3-form $\phi=\alpha+\mu\wedge\beta-\alpha_2$, which has $\dx\phi=-\beta^2-2\mu\alpha_1$ and $\dx*\phi=0$. The equations (\ref{flatcase}) deduced in coordinates for the flat case partly confirm our algebraic results of proposition (\ref{derivacoesexteriores}).

\subsection{Computing $\dx\phi$ and $\dx*\phi$}
\label{computing}

An orientable Riemannian 7-manifold with a $G_2$ structure $\phi$ admits a holonomy reduction to $G_2$ if, and only if, $\phi$ is parallel. Such condition being fulfilled gives rise to the concept of a $G_2$ \textit{manifold}. A theorem of A. Gray says this is equivalent to having $\phi$ harmonic. If $\dx\phi=0$, then the structure is called \textit{calibrated} and, if $\delta\phi=0$, the structure is called \textit{co-calibrated}, cf. \cite{Agri,Bryant2,FerGray,FriKaMoSe}.

We shall classify our \gwistor\ space and hence must compute several derivatives.
\begin{prop}\label{derivacoes}
For any vector field $X$ over $SM$:
\begin{meuenumerate}
\item $\nag_X\alpha=\frac{1}{4}\alpha\circ(\calR(X,\cdot)\wedge\Id\wedge\Id)= A_X\alpha$.
\item $\nag_X\alpha_1=\tfrac{1}{2}(X^v\lrcorner\inv{f}\vol_M)\circ\theta\wedge\Id\wedge\Id+ \tfrac{1}{2}\alpha\circ(\calR(X,\cdot)\wedge\theta\wedge\Id) -\tfrac{1}{2}\alpha\circ\,\theta (A_X+\tau_X)\wedge\Id\wedge\Id$.
\item $\nag_X\alpha_2=\tfrac{1}{2}(X^v\lrcorner\inv{f}\vol_M)\circ\theta\wedge\theta\wedge\Id+ \tfrac{1}{4}\alpha\circ(\calR(X,\cdot)\wedge\theta\wedge\theta)
- \alpha\circ\,\theta(A_X+\tau_X)\wedge\theta\wedge\Id$.
\item $\nag_X\alpha_3=\tfrac{1}{6}(X^v\lrcorner\inv{f}\vol_M)\circ\theta^3
-\tfrac{1}{2}\alpha\circ\,\theta (A_X+\tau_X)\wedge\theta^2$.
\item $\nag_X\mu=X^\flat\circ\theta +\mu\circ (A_X+\tau_X)$.
\item $\nag_X\beta=\beta\circ(\tfrac{1}{2}\calR(X,\cdot)-A_X-\tau_X\ \wedge\Id)$.
\item $\nag_X\theta^tU=\theta^tX^v-\tfrac{1}{2}\calR(X,\theta^tU)+(A+\tau)(X,\theta^tU).$
\end{meuenumerate}
\end{prop}
\begin{proof}
1. We have $\nag_XY_i=\Dstar_XY_i-\frac{1}{2}\calR(X,Y_i)+A_XY_i+\tau_XY_i$ for
any three vector fields $Y_1,Y_2,Y_3$ on $SM$ and it is easy to see $\Dstar\alpha=D^*\alpha$, hence
\begin{eqnarray*}
\lefteqn{\nag_X\alpha(Y_1,Y_2,Y_3) \ =\ \Dstar_X(U\lrcorner\,\inv{f}\vol_M)(Y_1,Y_2,Y_3)\,+}  \\
        & & +\frac{1}{2}\bigl(\alpha(\calR(X,Y_1),Y_2,Y_3) +\alpha(Y_1,\calR(X,Y_2),Y_3)+\alpha(Y_1,Y_2,\calR(X,Y_3))\bigr)\\
        &=&  (D^*_X\inv{f}\vol_M)(U,Y_1,Y_2,Y_3)+\inv{f}\vol_M(D^*_XU,Y_1,Y_2,Y_3)+  \\
        & & +\frac{1}{2}\bigl(\alpha(\calR(X,Y_1),Y_2,Y_3) +\alpha(\calR(X,Y_2),Y_3,Y_1) +\alpha(\calR(X,Y_3),Y_1,Y_2)\bigr).\qquad\ (*)
\end{eqnarray*}
The first term on the sum vanishes because $D\vol_M=0$; the second is $X^v\lrcorner\inv{f}\vol_M$ because $D^*_XU=X^v$. Since the $X^v$ and $Y_i^v$ are linearly dependent, this part also vanishes. For the third term, notice
\[ e^{j}(\calR(X,\cdot))=\langle e_{j},R^*(X,\cdot)U\rangle=2\langle A_Xe_{j},\cdot\rangle,  \]
for $4\leq j\leq6$, and so, if we see $\alpha=e^{456}$ as previously, then 
\[ \frac{1}{4}\alpha\circ(\calR(X,\cdot)\wedge\Id\wedge\Id)=
(A_Xe_4)^\flat\wedge e^{56}-(A_Xe_5)^\flat\wedge e^{46}+(A_Xe_6)^\flat\wedge e^{45}, \]
ie. $A_X$ acts as a derivation of $\alpha$. \\
2. Let $\tilde{\na}^g=\nag+\langle D^*\ ,U\rangle U$, ie. the Levi-Civita connection of $TM$. Since $\alpha$ and $\theta$ vanish when we take one direction proportional to $U$, the derivative of $\alpha_1$ we have to compute can be made with $\tilde{\na}^g$. Since $\tilde{\na}^g_X\theta=[A_X+\tau_X,\theta]$, the Leibniz rule yields
\begin{equation*}
 \begin{split}
 \nag_X(\tfrac{1}{2}\alpha\circ\theta\wedge\Id\wedge\Id) \ =\  \tfrac{1}{2}\nag_X\alpha\circ\,\theta\wedge\Id\wedge\Id +  \tfrac{1}{2}\alpha\circ\,\tilde{\na}^g_X\theta\wedge\Id\wedge\Id \\
=\ \tfrac{1}{2}(X^v\lrcorner\inv{f}\vol_M+A_X\alpha)\circ\theta\wedge\Id\wedge\Id -\tfrac{1}{2}\alpha\circ\,\theta (A_X+\tau_X)\wedge\Id\wedge\Id
 \end{split}
\end{equation*}
also because $A,\tau$ are horizontal valued. Taking computation in 1. into account, it is easy to see that three of the nine summands in 
\[ \tfrac{1}{2}(A_X\alpha)\circ\theta\wedge\Id\wedge\Id\,(Y_1,Y_2,Y_3)\:=\: \cyclic_{Y_1,Y_2,Y_3}A_X\alpha\:(\theta Y_1,Y_2,Y_3)\]
are 0 because $\calR(X,\theta Y_i)=0$. And we are left with the six summands of $\frac{1}{2}\alpha\circ(\calR(X,\ )\wedge\theta\wedge\Id)$, cf. (*). In sum, we found
\begin{equation*}
\begin{split}
\nag_X\alpha_1=\tfrac{1}{2}(X^v\lrcorner\inv{f}\vol_M)\circ\theta\wedge\Id\wedge\Id+\hspace{1.5cm}\\ +\tfrac{1}{2}\alpha\circ(\calR(X,\ )\wedge\theta\wedge\Id)- \tfrac{1}{2}\alpha\circ\,\theta (A_X+\tau_X)\wedge\Id\wedge\Id. 
\end{split}
\end{equation*}
3. Proceeding as above,
\begin{equation*}
 \begin{split}
 \nag_X(\tfrac{1}{2}\alpha\circ\theta\wedge\theta\wedge\Id) \ =\  \tfrac{1}{2}\nag_X\alpha\circ\,\theta\wedge\theta\wedge\Id +  \alpha\circ\,\tilde{\na}^g_X\theta\wedge\theta\wedge\Id \\
=\ \tfrac{1}{2}(X^v\lrcorner\inv{f}\vol_M+A_X\alpha)\circ\theta\wedge\theta\wedge\Id -\alpha\circ\,\theta (A_X+\tau_X)\wedge\theta\wedge\Id.
 \end{split}
\end{equation*}
We now have six null summands in nine, leaving us with
\[ \tfrac{1}{2}(A_X\alpha)\circ\theta\wedge\theta\wedge\Id= 
\tfrac{1}{4}\alpha\circ(\calR(X,\ )\wedge\theta\wedge\theta) . \]
4. Proceeding as before,
\begin{equation*}
 \begin{split}
 \nag_X\tfrac{1}{6}\alpha\circ\theta^3 \ =\  \tfrac{1}{6}\nag_X\alpha\circ\,\theta^3 +  \tfrac{1}{2}\alpha\circ\,\tilde{\na}^g_X\theta\wedge\theta^2 \hspace{2cm}\\
=\ \tfrac{1}{6}(X^v\lrcorner\inv{f}\vol_M+A_X\alpha)\circ\theta^3-\tfrac{1}{2}\alpha\circ\,\theta (A_X+\tau_X)\wedge\theta^2\\
=\ \tfrac{1}{6}X^v\lrcorner\inv{f}\vol_M\circ\theta^3
-\tfrac{1}{2}\alpha\circ\,\theta (A_X+\tau_X)\wedge\theta^2.\ \ 
 \end{split}
\end{equation*}
5. A straightforward computation:
\begin{eqnarray*}
  \nag_X\mu(Y)&=& X(\mu Y)-\mu(\nag_XY)\\
&=&\langle D^*_XU,\theta Y\rangle +\langle U,D^*_X(\theta Y)\rangle-\langle U,\theta(D^*_XY)\rangle-\langle U,\theta(A_XY+\tau_XY)\rangle\\
&=&  \langle X,\theta Y\rangle-\mu((A+\tau)_XY).
\end{eqnarray*}
6. This is the simple consequence of $\Dstar\beta=D^*\beta=0$. Computation through $\nag\beta=g\circ\nag\theta\wedge\Id$ is equal in length.\\
7. From $\Dstar_X\theta^tU=\theta^tD^*_XU=\theta^tX^v$, we conclude the result.
 \end{proof}
Recall the Ricci tensor of $M$ is defined by $r(X,Y)=\Tr{R^D(\cdot,X)Y}$. It determines an endomorphism $\ric\in\Omega^0(\End{TM})$ satisfying $r(X,Y)=\langle X,\ric Y\rangle,\ \,\forall X,Y\in TM$. Recall $r$ is symmetric if $T^D=0$, ie. in the case of $D$ being the Levi-Civita connection of $M$.

On $SM$ we shall denote by $\underline{r}$ the function $r(U,U)$ and set $\rho=(\ric U)^\flat\in\Omega^0({\cal V}^*)$. It is a 1-form on $SM$, vanishing on $\hnab$ and restricted to vertical tangent directions. One may view $\rho$ as the vertical lift of $r(\ ,U)$. Consider the usual frame, used in the proof of proposition \ref{bse}, and let $R^*_{ijkl}=\langle R(e_i,e_j)e_k,e_l\rangle$. Then
\begin{equation}\label{roUric}
 \rho=\sum_{i,k=1}^3R_{ki0k}e^{i+3}\qquad\mbox{and}\qquad\underline{r}=\sum_jR_{j00j}.
\end{equation}
In the following we abbreviate $T^*,R^*$ for $T,R$ respectively; these are both totally horizontal tensors. Also let $\bx$ denote the operator of cyclic sum on 3-tensors. On \gwistor\ space we also define the following scalar functions
\begin{equation}\label{deflem}
 l=R_{1230}+R_{2310}+R_{3120}\qquad\quad \mbox{and}\qquad\quad  m=\Tr{T(e_0,\ )}
\end{equation}
and 3-forms
\begin{equation}\label{defvarrhoegama}
 \varrho=\bx\,\mu(R(\ ,\ )\ )\qquad\quad \mbox{and}\qquad\quad \sigma=\bx\,\beta(T(\ ,\ ),\ ).
\end{equation}
Of course all these last four, $l,m,\varrho,\sigma$, vanish in the torsion free context. Notice $l=\varrho_{123}$ does not depend on the choice of the orthonormal frame $e_1,e_2,e_3$ which together with $e_0=u$ is positively oriented. The next result includes some new definitions.
\begin{prop}\label{derivacoesexteriores}
We have the following formulae:
\begin{meuenumerate}
\item $\dx\alpha=\calR\alpha$.
\item $\dx\alpha_1=3\mu\alpha+\calR\alpha_1+T\alpha_1$\, and \,$\mu\wedge\calR\alpha_1=-\vol^*_M\wedge\rho$.
\item $\dx\alpha_2=2\mu\alpha_1-\underline{r}\,\vol^*_M+T\alpha_2$.
\item $\dx\alpha_3=\mu\alpha_2+m\vol^*_M.$
\item $\dx\,\vol_M^*=0$.
\item $\dx\mu=-\beta+\mu(T)$ and \,$\delta\mu=-m$.
\item $\dx\beta=\dx(\mu T)=\varrho+\sigma$.
\item If $D$ is torsion free, then $\dx\beta=0,\ \,\dx\mu=-\beta$\, and \,$\delta\mu=0$.
\end{meuenumerate}
\end{prop}
\begin{proof}
1. We continue to use the frame $e_0,\ldots,e_6$. Let $R_{ijkl}$ be as above. Now
\begin{eqnarray*}
 \dx\alpha=\sum_{i=0}^6e^i\wedge\nag_{e_i}\alpha=\frac{1}{4}\sum_{i=0}^6e^i\wedge \alpha\circ(\calR(e_i,\cdot)\wedge\Id\wedge\Id)=\hspace{1.5cm}\\
=\sum_{i=0}^6e^i\wedge\bigl((A_{e_i}e_4)^\flat\wedge e^{56}+(A_{e_i}e_5)^\flat\wedge e^{64}+(A_{e_i}e_6)^\flat\wedge e^{45}\bigr).
\end{eqnarray*}
Since $\langle A_ie_4,e_j\rangle-\langle A_je_4,e_i\rangle=\langle R^*(e_i,e_j)U,e_4\rangle=R_{ij01}$, we find
\begin{equation}\label{dalpha}
 \dx\alpha=\sum_{0\leq i<j\leq 3}R_{ij01}e^{ij56}+R_{ij02}e^{ij64}+R_{ij03}e^{ij45}\ =\ \calR\alpha.
\end{equation}
2. We start by computing
\begin{eqnarray*}
& &\tfrac{1}{2}\sum_{i=0}^6e^i\wedge (e_i^v\lrcorner\inv{f}\vol_M)\circ\theta\wedge\Id\wedge\Id\ =\ \tfrac{1}{2}e^4\wedge \inv{f}\vol_M(e_4,\theta\wedge\Id\wedge\Id)+\\
& &\quad +\tfrac{1}{2}e^5\wedge \inv{f}\vol_M(e_5,\theta\wedge\Id\wedge\Id)+\tfrac{1}{2}e^6\wedge \inv{f}\vol_M(e_6,\theta\wedge\Id\wedge\Id)\\
& & \ =\ -e^4\wedge e^{056}+e^5\wedge e^{046}-e^6\wedge e^{045}\\
& & \ =\ 3\mu\alpha.
\end{eqnarray*}
Now since $e^{j+3}\theta=e^j$ and $A$ and $\tau$ have only horizontal values, we have
\begin{eqnarray*}
& &\frac{1}{2}\sum_{i=0}^6e^i\wedge \bigl( 
(A_i\alpha)\circ\theta\wedge\Id\wedge\Id -\alpha\circ\,\theta A_i\wedge\Id\wedge\Id-\alpha\circ\,\theta\tau_i\wedge\Id\wedge\Id\bigr)\\
&=&\sum_{i=0}^3 e^i\wedge\bigl((A_{e_i}e_4)^\flat\wedge(e^{26}+e^{53})- (A_{e_i}e_5)^\flat\wedge(e^{16}+e^{43}) +(A_{e_i}e_6)^\flat\wedge(e^{15}+e^{42})\bigr)\\
& &\ \ \ -\frac{1}{2}\sum_{i,j=0}^6e^i\wedge\alpha\circ(\theta (A_{e_i}e_j)e^j\wedge\Id\wedge\Id) -\sum_{i,j=0}^3 \tau_{ij1}e^{ij56}+\tau_{ij2}e^{ij64}+\tau_{ij3}e^{ij45}.
\end{eqnarray*}
Here we have used $\tau_i=\tau_{ijk}e_ke^j$. The second sum above is a contraction of a symmetric tensor $A$ within a skew tensor, so its contribution is null. Also, using the symmetries of $\tau$, $\tau_{ijk}-\tau_{jik}=-T_{ijk}$, we find
\begin{equation}\label{dalpha1}
\begin{split}
\dx\alpha_1\ =\ \sum_{i=0}^6e^i\wedge\nag_{e_i}\alpha_1 \hspace{7cm}\\
 \ =\ 3\mu\alpha+\sum_{0\leq i<j\leq3} R_{ij01}(e^{ij26}+e^{ij53})
-R_{ij02}(e^{ij16}+e^{ij43})+\\
 +R_{ij03}(e^{ij15}+e^{ij42}) 
+\sum_{0\leq i<j\leq3}T_{ij1}e^{ij56}+T_{ij2}e^{ij64}+T_{ij3}e^{ij45}\\
\ =\  3\mu\alpha+\calR\alpha_1+T\alpha_1.\hspace{6.5cm}
\end{split}
\end{equation}
Now if we couple $\mu$ with the first sum above, then we get
\begin{equation*}
\mu\wedge\calR\alpha_1=-(R_{1301}+R_{2302})e^{01236}-(R_{1201}+R_{3203})e^{01235} -(R_{2102}+R_{3103})e^{01234}
\end{equation*}
just by using the two skew-symmetries in $R$ which do not depend of $T^D$. Since 
$\rho=\sum_{i=1}^3r(e_i,e_0)e^{i+3}$, it is immediate to conclude $\mu\wedge\calR\alpha_1=-\rho\wedge\vol^*_M$.\\
3. In this case we do as above a previous computation for the part with $\frac{1}{2}\inv{f}\vol_M$. It is the same expression but with two $\theta$ instead of two $\Id$ and thus $e_0$ may enter twice. So we get the non-vanishing terms
\begin{equation*}
 \begin{split}
  -e^{4026}+e^{4035}-e^{5034}+e^{5016}-e^{6015}+e^{6024} \\
  =\ 2e^0(e^{264}+e^{345}+e^{156})\ =\ 2\mu\alpha_1.\hspace{1.8cm}
 \end{split}
\end{equation*}
Hence
\begin{eqnarray*}
\dx\alpha_2 &=& \sum_{i=0}^6e^i\wedge \tfrac{1}{2}(e_i^v\lrcorner\inv{f}\vol_M)\circ\theta\wedge\theta\wedge\Id+ \\
& &+e^i\wedge \bigl( \tfrac{1}{2}(A_i\alpha)\circ\theta\wedge\theta\wedge\Id -\alpha\circ\,\theta A_i\wedge\theta\wedge\Id-\alpha\circ\,\theta\tau_i\wedge\theta\wedge\Id\bigr).
\end{eqnarray*}
And after careful inspection
\begin{equation}\label{dalpha2}
 \begin{split}
\dx\alpha_2 \ =\ 2\mu\alpha_1+\sum_{0\leq i<j\leq3}R_{ij01}e^{ij23}+R_{ij02}e^{ij31}+R_{ij03}e^{ij12} \hspace{2cm}\\
 +\sum_{0\leq i<j\leq3}T_{ij1}(e^{ij26}+e^{ij53}) -T_{ij2}(e^{ij16}+e^{ij43})+T_{ij3}(e^{ij15}+e^{ij42})\\
=\ 2\mu\alpha_1-\underline{r}\vol^*_M+T\alpha_2\hspace{3cm}
\end{split}
\end{equation}
since it is easy to see the expression with $R$ simplifies to $-r(U,U)\vol^*_M$.\\
4. We start by noticing that
\[ \tfrac{1}{6}(e^4\lrcorner\inv{f}\vol_M)\circ\theta^3= \tfrac{1}{6}e^4(-e^{023}+e^{032}+e^{203}-e^{230}+e^{320}-e^{302})=\mu e^{234}.\]
Then adding terms with $e_5$ and $e_6$ in place of $e_4$, gives $\mu\alpha_2$. Because of the symmetry of $A$ we are left with
\begin{equation}\label{dalpha3}
 \begin{split}
\dx\alpha_3 \ =\ \mu\alpha_2-\tfrac{1}{2} \sum_{i=0}^6e^i\wedge\alpha\circ(\theta\tau_i\wedge\theta\wedge\theta) \hspace{3cm}\\
 =\ \mu\alpha_2+ \tfrac{1}{2}\sum_{i<j\leq3}e^{ij}\bigl(T_{ij1}(e^{23}-e^{32})+T_{ij2}(e^{31}-e^{13}) +T_{ij3}(e^{12}-e^{21})\bigr)\\
=\ \mu\alpha_2+(T_{011}e^{0123}+T_{022}e^{0231}+T_{033}e^{0312})=\mu\alpha_2+m\vol^*_M.\ \ 
\end{split}
\end{equation}
5. Indeed, $\dx\,\vol_M^*=\dx f^*\vol_M=0$.\\
6. Since $A$ is symmetric and $\tau_{ijk}-\tau_{jik}=-T_{ijk}$, we get
\begin{eqnarray*}
 \dx\mu(X,Y)&=&(\nag_X\mu)Y-(\nag_Y\mu)X  \\
&=&\langle X,\theta Y\rangle-\mu(\tau_XY)-\langle Y,\theta X\rangle+\mu(\tau_YX)\\
&=&-\beta(X,Y)+\mu(T(X,Y)).
\end{eqnarray*}
Furthermore, $\delta\mu=-*\dx *\mu=\frac{1}{6}*\dx\beta^3=\frac{1}{2}*\beta^2\dx\beta$. That this is $-m$ follows next.\\
7. It is well known that for any connection $D$ we may extend the de Rham operator $\dx$ to the respective vector bundle. Also it is well known that $\dx^D\Id=T^D$ and $\dx^DT^D=\bx R^D$, the first Bianchi identity. Recall $\bx$ denotes the operator of cyclic sum on 3-tensors. In our case it is easy to see that $\dx^{\Dstar}T=\dx^{D^*}T=\bx R$. A straightforward proof would be through a commuting rule $\dx^{D^*}f^*=f^*\dx^D$. Hence
\begin{eqnarray*}
 \dx(\mu(T))(X,Y,Z)&=& \bx\bigl( X(\mu(T(Y,Z)))-\mu(T([X,Y],Z))\bigr) \\
&=& \mu(\dx^{\Dstar}T(X,Y,Z))+\bx \bigl(X(\mu(T(Y,Z)))-\mu(\Dstar_X(T(Y,Z)))\bigr)\\
&=& \mu(\bx\,R(X,Y,Z))+\bx\:(D^*_X\mu)(T(Y,Z))\\
&=& \bx\bigl( \mu(R(X,Y)Z)+\langle X,\theta T(Y,Z)\rangle\bigr)\\
&=& \bx\bigl( \mu(R(X,Y)Z)+\beta(T(Y,Z),X)\bigr)
\end{eqnarray*}
since $T$ is horizontal; finally we use $\dx\beta=\dx(\mu T)$. Notice $R$ is $\hnab$-valued. On the usual frame,
\begin{equation}\label{dbeta}
 \dx\beta=
\sum_{0\leq i<j<k\leq3}\bx R_{ijk0}\,e^{ijk}+\sum_{0\leq i<j\leq3}T_{ij1}e^{ij4}+T_{ij2}e^{ij5}+T_{ij3}e^{ij6}
=\varrho+\sigma.
\end{equation}
Since $\beta^2=-2(e^{1245}+e^{1346}+e^{2356})$ we get $*\beta^2\dx\beta=-2(T_{011}+T_{022}+T_{033})=-2m$.\\
8. This is immediate.
\end{proof}
Notice formulas (\ref{dalpha}), (\ref{dalpha1}) and (\ref{dalpha2}) are defining expressions of $\calR\alpha,\ \calR\alpha_1,\ T\alpha_1$ and $T\alpha_2$. Furthermore, (\ref{defvarrhoegama}) is rewritten in (\ref{dbeta}). From now on we drop the symbol $\wedge$, for brevity. Also we let $\vol=\vol^*_M$.
\begin{prop}
The $G_2$ structure forms $\phi=\alpha+\mu\beta-\alpha_2$ and $*\phi=\vol-\frac{1}{2}\beta^2-\mu\alpha_1$ of $SM$ satisfy
\begin{equation}\label{dphiedestrelaphi1}
\dx\phi=\calR\alpha+(\underline{r}-l)\vol-\beta^2-2\mu\alpha_1+(\mu T)\beta-\mu\sigma-T\alpha_2
\end{equation}
and
\begin{equation}\label{dphiedestrelaphi2}
    \dx*\phi=-\beta\varrho-\rho\vol-\beta\sigma-(\mu T)\alpha_1+\mu(T\alpha_1).
\end{equation}
\end{prop}
\begin{proof}
In view of proposition \ref{derivacoesexteriores}, for $\phi$ only $\mu\dx\beta=\mu\varrho+\mu\sigma$ offers some challenge. Indeed $\mu\varrho=\sum_{i<j<k}\bx\,R_{ijk0}e^{0ijk}=l\vol^*_M$. For $*\phi$ we have
\begin{equation*}
 \begin{split}
  \dx*\phi=-(\dx\beta)\beta-(\dx\mu)\alpha_1+\mu\dx\alpha_1\hspace{3cm}\\
= -(\varrho+\sigma)\beta+(\beta-\mu T)\alpha_1+\mu(3\mu\alpha+\calR\alpha_1+T\alpha_1).\\
 \end{split}
\end{equation*}
And we had seen that $\beta\alpha_1=\mu(3\mu\alpha)=0$ and $\mu\calR\alpha_1=-\rho\vol$.
\end{proof}

\section{The torsion forms}
\label{Thetorsions}

\subsection{Representation of $G_2$}
\label{ttf}

The irreducible decomposition of $\Lambda^*\R^7$ as a $G_2$-module may be seen in well known references such as \cite{Bryant2,FerGray}. Since the star operator commutes with the group product, the problem resumes to say degrees 2 and 3 (or 4 and 5). We have
\begin{equation}\label{decomp} 
\Lambda^2=\Lambda^2_{7}\oplus\Lambda^2_{14},\hspace{2cm} \Lambda^3=\Lambda^3_{1}\oplus\Lambda^3_{7}\oplus\Lambda^3_{27},
\end{equation}
where $\Lambda^2_{7}=\{\gamma\in\Lambda^2:\,\ \gamma\wedge\phi=-2*\gamma\}$, 
\,$\Lambda^2_{14}=\{\gamma\in\Lambda^2:\,\ \gamma\wedge\phi=*\gamma\}\simeq\g_2$, $\Lambda^3_1=\R\phi$, \,$\Lambda^3_7=\{*(\gamma\wedge\phi):\,\ \gamma\in\Lambda^1\}$, $\Lambda^3_{27}=\{\gamma\in\Lambda^3:\,\ \gamma\wedge\phi=\gamma\wedge*\phi=0\}$, with the indices below standing for the dimensions.

Passing to differential forms, the unique components of $\dx\phi$ and $\dx*\phi$ on the representation subspaces are called the \textit{torsion forms} of the $G_2$ structure in question. Fortunately one of them occurs in two places:
\begin{equation}\label{torsoes0}
 \dx*\phi=\tau_1\wedge*\phi+*\tau_2,   \hspace{1.7cm} \dx\phi=\tau_0*\phi+\frac{3}{4}\tau_1\wedge\phi+*\tau_3,
\end{equation} 
with $\tau_i\in\Omega^i,\ \tau_2\in\Omega^2_{14},\ \tau_3\in\Omega^3_{27}$. Hence there are in principle sixteen classes of $G_2$ structures. There is a most valuable formula in \cite{FerGray}: if we are given $\kappa_5\in\Omega^5$, then the respective $\kappa_1\in\Omega^1$ for the first decomposition in (\ref{torsoes0}) is 
\begin{equation}\label{formdek1}
 \frac{1}{3}*(*\kappa_5\wedge*\phi).
\end{equation}
$\tau_1$ is known as the \textit{Lee form}. Following the terminology of \cite{Agri,FriIva2}, a $G_2$ structure is called \textit{balanced} if $\tau_1$ vanishes, \textit{integrable} if $\tau_2=0$, \textit{co-calibrated} if $\tau_1=\tau_2=0$ and \textit{co-calibrated of pure type $W_3$} if $\tau_0=\tau_1=\tau_2=0$. If $\tau_1=\tau_2=\tau_3=0$, then the structure is called \textit{nearly parallel}. In this case $\tau_0$ is constant. Reciprocally, if $\dx\phi=\tau_0*\phi$ with $\tau_0\neq0$ constant, then the structure is nearly parallel. Of course, we speak of $G_2$ manifold or \textit{parallel} structure, if also $\tau_0=0$.

Notice the wedge on 1-forms with $\phi$ or $*\phi$ is a $G_2$-equivariant monomorphism. We are now going to find the four torsion forms $\tau_0,\tau_1,\tau_2,\tau_3$ which will classify our $G_2$ structure. We ask the reader to distinguish the meaning of the word torsion, as $\tau_i$ or $T^D$, from the context.

\subsection{Computation of torsion tensors $\tau_1$ and $\tau_2$}

We start by $\dx*\phi$. It is the sum $\dx*\phi=\dastC+\dastT$ where
\begin{equation}
 \dastC=-\beta\varrho-\rho\vol
\end{equation}
and
\begin{equation}
 \dastT=-\beta\sigma-(\mu T)\alpha_1+\mu(T\alpha_1).
\end{equation}
We remark the first part only involves $R^D$ and it is a multiple of $\vol$, whilst in the second only $T^D$ appears and each summand has exactly three indices $\leq3$ in the usual frame. Because of this nature we choose to give the torsion forms separately. Let
\begin{equation}
 \left\{\begin{array}{l}
\tilde{R}_1=R_{0230}+R_{3020}+R_{2102}+R_{3103}\\
\tilde{R}_2=-R_{0130}-R_{3010}+R_{1201}+R_{3203}\\
\tilde{R}_3=R_{0120}+R_{2010}+R_{1301}+R_{2302}
\end{array}\right. 
\end{equation}
and consider the vector field
\begin{equation}
 \tilde{R}=\tilde{R}_1e_4+\tilde{R}_2e_5+\tilde{R}_3e_6.
\end{equation}
Then $\dastC=-\tilde{R}^\flat\,\vol$, as we shall see. Hence $\tilde R$ does not depend on the frame; in other words it is a global vector field. For the torsion-free case, $T^D=0$, we have $\tilde{R}^\flat=\rho$ and $\tilde{R}_i=(\ric U)_i$. In the following let $\astil$ be the pull-back of the Hodge operator on $M$ and consider the $\astil$-anti-selfdual 2-forms $\hat\omega_1=e^{01}-e^{23},\ \hat\omega_2=e^{02}+e^{13},\ \hat\omega_3=e^{03}-e^{12}$.
\begin{prop}\label{tau1etau2}
 We have $\dastC=\tau_{1,\mathrm{curv}}\wedge*\phi+*\tau_{2,\mathrm{curv}}$ with
\begin{equation}
 \tau_{1,\mathrm{curv}}=-\frac{1}{3}\vol^\sharp\lrcorner\dastC=-\frac{1}{3}\tilde{R}^\flat
\end{equation}
and
\begin{equation}
\tau_{2,\mathrm{curv}}=-\frac{2}{3}\tilde{R}\lrcorner\alpha+\frac{1}{3}\sum_i\tilde{R}_i\hat\omega_i
= \frac{1}{3}\tilde{R}\lrcorner(\phi-3\alpha).
\end{equation}
Moreover, $\delta\phi=\tilde{R}\lrcorner\alpha$.
In particular, if $D$ is torsion free, then $SM$ is co-calibrated if, and only if, $(M,g)$ is an Einstein manifold.
\end{prop}
\begin{proof}
We have
\begin{eqnarray*}
 & & \dastC\ =\ -\beta\varrho-\rho\vol\ =    \\
&& \ \ \ \ =\ (e^{41}+e^{52}+e^{63})\sum\bx R_{ijk0}e^{ijk}-\bigl((R_{1301}+R_{2302})e^{6}+ \bigr.  \\
& &\ \bigl.\hspace{1cm}+(R_{1201}+R_{3203})e^{5}+(R_{2102}+R_{3103})e^{4}\bigr)\vol\:=\:-\sum \tilde{R}_ie^{i+3}\vol.
\end{eqnarray*}
And so $\delta\phi=-*\dx*\phi$ follows. Since $\beta^2$ and $\alpha_1$ have two degrees above 3,
\[ \tau_{1,\mathrm{curv}}=\tfrac{1}{3}*(*\dastC*\phi)= -\tfrac{1}{3}*((\tilde{R}_1e^{56}+\tilde{R}_2e^{64}+\tilde{R}_3e^{45})\vol) =-\tfrac{1}{3}\sum \tilde{R}_ie^{i+3} .\]
Finally, \,$\tau_{2,\mathrm{curv}}\ =$
\begin{eqnarray*}
 &=&*\dastC-*(\tau_{1,\mathrm{curv}}*\phi)\\
&=&(-\tilde{R}_i+\tfrac{1}{3}\tilde{R}_i)*(e^{i+3}\vol)+\tfrac{1}{3}*(\tilde{R}_ie^{i+3}(e^{1245}+e^{1346}+\\
& & \hspace{4.2cm}+e^{2356}-e^{0156}-e^{0264}-e^{0345}))\\
&=&-\tfrac{2}{3}(\tilde{R}_1e^{56}+\tilde{R}_2e^{64}+\tilde{R}_3e^{56})+\tfrac{1}{3}(\tilde{R}_1(e^{01}-e^{23})+\tilde{R}_2(e^{02}+e^{13})+\tilde{R}_3(e^{03}-e^{12}))  \\
&=&-\tfrac{2}{3}\tilde{R}\lrcorner\alpha+\tfrac{1}{3}\tilde{R}_i\hat\omega_i.
\end{eqnarray*}
The formula is further simplified using $\phi=\alpha+\hat\omega_ie^{i+3}$.
In case $T^D=0$, then clearly $\delta\phi=\ric U\lrcorner\alpha$ and this vanishes if, only if, $\ric U$ is a multiple of $U$.
\end{proof}
Now we proceed with the component $\dastT=-\beta\sigma-(\mu T)\alpha_1+\mu(T\alpha_1)$. In order to describe the solutions we need the \textit{adapted} frame of $TSM$ which we have been using. Let
\begin{equation}\label{osZes}
 Z_{ijk}=T_{ijj}+T_{ikk}+T_{jkl}\qquad\mbox{where}\ (i,j,k,l)\ \mbox{is a direct ordering}
\end{equation}
and let
\begin{equation}\label{osWos}
 W_{i}=Z_{ijk}+Z_{ilj}+Z_{ikl}
\end{equation}
ie. the sum of $Z_{i\ldots}$ such that $(i,j,k,l)$ is a direct ordering. Then of course
\begin{equation}\label{osWos2}
 W_i=2T_{ijj}+2T_{ikk}+2T_{ill}+\bx T_{jkl}.
\end{equation}
\begin{prop}\label{tau1tau2compodetorsao}
We have $\dastT=\tau_{1,\mathrm{tors}}\wedge*\phi+*\tau_{2,\mathrm{tors}}$ with
\begin{equation}\label{tau1t}
 \tau_{1,\mathrm{tors}}=\frac{1}{3}\sum_{i=0}^3W_ie^i
\end{equation}
and
\begin{equation}\label{tau2t}
 \begin{split}
\tau_{2,\mathrm{tors}}\:=\:
(\tfrac{1}{3}W_3-Z_{321})e^{06}-(\tfrac{1}{3}W_2-Z_{230})e^{16}
+(\tfrac{1}{3}W_1-Z_{103})e^{26}-(\tfrac{1}{3}W_0-Z_{012})e^{36}\\
 +(\tfrac{1}{3}W_2-Z_{213})e^{05}+(\tfrac{1}{3}W_3-Z_{302})e^{15}
-(\tfrac{1}{3}W_0-Z_{031})e^{25}-(\tfrac{1}{3}W_1-Z_{120})e^{35}\\
 +(\tfrac{1}{3}W_1-Z_{132})e^{04}-(\tfrac{1}{3}W_0-Z_{023})e^{14}
-(\tfrac{1}{3}W_3-Z_{310})e^{24}+(\tfrac{1}{3}W_2-Z_{201})e^{34}.
 \end{split}
\end{equation}
\end{prop}
\begin{proof}
From (\ref{dalpha1}) and (\ref{dbeta}) we find
\begin{eqnarray*}
\dastT &=& \sum_{0\leq i<j\leq3}-(T_{ij1}e^{ij4(25+36)}+T_{ij2}e^{ij5(14+36)}+T_{ij3}e^{ij6(14+25)})\\
& &    \ \ -T_{ij0}e^{ij(156+264+345)}+T_{ij1}e^{0ij56}+T_{ij2}e^{0ij64}+T_{ij3}e^{0ij45}\\
&=& \ (\sum\:T_{ij1}e^{ij2}-T_{ij2}e^{ij1}-T_{ij0}e^{ij3}+T_{ij3}e^{0ij})e^{45}\\
& & +(\sum\:-T_{ij1}e^{ij3}+T_{ij3}e^{ij1}-T_{ij0}e^{ij2}+T_{ij2}e^{0ij})e^{64}\\
& & +(\sum\:T_{ij2}e^{ij3}-T_{ij3}e^{ij2}-T_{ij0}e^{ij1}+T_{ij1}e^{0ij})e^{56}.
\end{eqnarray*}
This defines 3-forms $\Theta_i$ such that 
\[ *\dastT=\astil\Theta_1e^6+\astil\Theta_2e^5+\astil\Theta_3e^4.  \]
where $\astil$ is the pull-back of the Hodge operator on $M$. Notice the orientation: we pertain to the canonical operator. For instance, $\astil e^1=-e^{023}$, so that $*e^1=(\astil e^1)\alpha$.
Now it is easy to see
\[ *(\dastT)*\phi=\bigl((\astil\Theta_1)(e^{12}-e^{03})+(\astil\Theta_2)(e^{31}-e^{02})+ 
 (\astil\Theta_3)(e^{23}-e^{01}) \bigr)\alpha . \]
Let $\Theta_1=a_0e^{123}+a_1e^{023}+a_2e^{013}+a_3e^{012}$, and define analogously coefficients $b_i,c_i$, $i=0,\ldots,3$, for $\Theta_2,\Theta_3$. Now we compute
\begin{equation*}
\begin{split}
 \astil\bigl((\astil(a_0e^{123}+a_1e^{023}+a_2e^{013}+a_3e^{012}))(e^{12}-e^{03})\bigr) \\
=\astil\bigl((-a_0e^{0}+a_1e^{1}-a_2e^{2}+a_3e^{3})(e^{12}-e^{03})\bigr)\\
=-a_0e^{3}-a_1e^{2}-a_2e^{1}-a_3e^{0},
\end{split}
\end{equation*}
\begin{equation*}
\begin{split}
 \astil\bigl((\astil(b_0e^{123}+b_1e^{023}+b_2e^{013}+b_3e^{012}))(e^{31}-e^{02})\bigr) \\
=\astil\bigl((-b_0e^{0}+b_1e^{1}-b_2e^{2}+b_3e^{3})(e^{31}-e^{02})\bigr)\\
=-b_0e^{2}+b_1e^{3}+b_2e^{0}-b_3e^{1},
\end{split}
\end{equation*}
\begin{equation*}
\begin{split}
 \astil\bigl((\astil(c_0e^{123}+c_1e^{023}+c_2e^{013}+c_3e^{012}))(e^{23}-e^{01})\bigr) \\
=\astil\bigl((-c_0e^{0}+c_1e^{1}-c_2e^{2}+c_3e^{3})(e^{23}-e^{01})\bigr)\\
=-c_0e^{1}-c_1e^{0}+c_2e^{3}+c_3e^{2}.
\end{split}
\end{equation*}
And we deduce and verify:
\begin{equation*}
 \begin{array}{l}
  a_0=T_{311}+T_{322}-T_{120}=Z_{321}\\
  a_1=T_{301}+T_{200}+T_{233}=Z_{230}\\
  a_2=T_{032}-T_{010}+T_{133}=Z_{103}\\
  a_3=T_{011}+T_{022}+T_{123}=Z_{012}
 \end{array}
\end{equation*}
\begin{equation*}
 \begin{array}{l}
  b_0=T_{211}+T_{233}+T_{130}=Z_{213}\\
  b_1=-T_{021}-T_{300}-T_{322}=-Z_{302}\\
  b_2=-T_{011}-T_{033}+T_{132}=-Z_{031}\\
  b_3=-T_{023}-T_{010}+T_{122}=Z_{120}
 \end{array}
\end{equation*}
\begin{equation*}
 \begin{array}{l}
  c_0=T_{122}+T_{133}-T_{230}=Z_{132}\\
  c_1=T_{022}+T_{033}+T_{231}=Z_{023}\\
  c_2=T_{012}-T_{300}-T_{311}=-Z_{310}\\
  c_3=-T_{013}-T_{200}-T_{211}=-Z_{201}.
 \end{array}
\end{equation*}
Finally we achieve
\begin{eqnarray*}
3\tau_{1,\mathrm{tors}}&=& *(*(\dastT)*\phi)\\
&=&*\bigl(\bigl((\astil\Theta_1)(e^{12}-e^{03})+(\astil\Theta_2)(e^{31}-e^{02})+ 
 (\astil\Theta_3)(e^{23}-e^{01})\bigr)\alpha\bigr)\\
&=&-\astil\bigl((\astil\Theta_1)(e^{12}-e^{03})+(\astil\Theta_2)(e^{31}-e^{02})+ 
 (\astil\Theta_3)(e^{23}-e^{01})\bigr)\\
&=& (a_3+c_1-b_2)e^0+(a_2+b_3+c_0)e^1+(a_1+b_0-c_3)e^2+(a_0-b_1-c_2)e^3
\end{eqnarray*}
and the first part of the result follows since we notice these are precisely the $W_i$ coefficients. For the second part we have 
\begin{eqnarray*}
 \tau_{2,\mathrm{tors}}&=&*(\dastT-\tfrac{1}{3}W_ie^i*\phi)\\
 &=&\astil\Theta_1e^6+\astil\Theta_2e^5+\astil\Theta_3e^4-\tfrac{1}{3}*\bigl(W_0(e^{01245}+e^{01346}+e^{02356})+ W_1(e^{12356}+  \bigr.\\
& & \bigl. + e^{01264}+e^{01345})+W_2(e^{02345}-e^{12346}-e^{01256})
+W_3(e^{12345}-e^{01356}-e^{02364})\bigr)\\
&=& (\tfrac{1}{3}W_3-a_0)e^{06}+(a_1-\tfrac{1}{3}W_2)e^{16}+
(\tfrac{1}{3}W_1-a_2)e^{26}+(a_3-\tfrac{1}{3}W_0)e^{36}+\\
& & +(\tfrac{1}{3}W_2-b_0)e^{05}+(b_1+\tfrac{1}{3}W_3)e^{15}+
(-\tfrac{1}{3}W_0-b_2)e^{25}+(b_3-\tfrac{1}{3}W_1)e^{35}+\\
& & +(\tfrac{1}{3}W_1-c_0)e^{04}+(c_1-\tfrac{1}{3}W_0)e^{14}+
(-\tfrac{1}{3}W_3-c_2)e^{24}+(c_3+\tfrac{1}{3}W_2)e^{34}
\end{eqnarray*}
and this is the result by analysis of the formulae previously deduced.
\end{proof}

\subsection{Computation of torsion tensors $\tau_0$ and $\tau_3$}

Here we deduce the $G_2$-torsion tensores of $\dx\phi$. To start with it proves useful to see the following.
\begin{prop}\label{mualpha1beta2voldecompostos}
We have the following $G_2$-decompositions all with no $\Omega_7^1$ part:
\begin{meuenumerate}
\item $2\mu\alpha_1=-\frac{6}{7}*\phi+*\frac{1}{7}(6\alpha+6\mu\beta+8\alpha_2)$.
\item $\beta^2=-\frac{6}{7}*\phi+*\frac{1}{7}(6\alpha-8\mu\beta-6\alpha_2)$.
\item $\vol=\frac{1}{7}*\phi+*\frac{1}{7}(6\alpha-\mu\beta+\alpha_2)$.
\end{meuenumerate}
\end{prop}
\begin{proof}
The formulas follow from the decompositions $2\mu\alpha_1=-\frac{6}{7}*\phi+*(\frac{6}{7}\phi+2\alpha_2)$,  $\beta^2=-\tfrac{6}{7}*\phi+*(*\beta^2+\tfrac{6}{7}\phi)$ and $\vol=\tfrac{1}{7}*\phi+*(\alpha-\tfrac{1}{7}\phi)$, which satisfy the required equations. For example, the first $\tau_3$ verifies clearly $\tau_3\phi=0$ and $\tau_3*\phi=6\Vol-2\alpha_2*\alpha_2=0$.
\end{proof}
\begin{prop}
 In the decomposition of $\dx\phi$ we have $\tau_0=\frac{2}{7}(\underline{r}-l+6)$.
\end{prop}
\begin{proof}
 Notice the wedge of a 4-form, as $\dx\phi$, with $\phi$ gives $7\tau_0\Vol_{SM}$. Indeed, the $G_2$-kernel of such map must contain $*\Omega^3_{7}\oplus*\Omega^3_{27}$. Now, since $((\mu T)\beta-\mu\sigma-T\alpha_2)\phi=0$ because its summands are either \textit{too heavy} or do not reach $e^{456}$, we get
\begin{eqnarray*}
 \lefteqn{ (\dx\phi)\phi\ =}\\ &=&(\calR\alpha-\beta^2+(\underline{r}-l)\vol-2\mu\alpha_1)(\alpha+\mu\beta-\alpha_2)\\
&=&(\calR\alpha)\mu\beta-(\calR\alpha)\alpha_2-\mu\beta^3+(\underline{r}-l)\vol\alpha +2\mu\alpha_1\alpha_2\\
&=& (R_{2301}-R_{1302}+R_{1203}-R_{0303}-R_{0101}-R_{0202}+6+\underline{r}-l+6)\Vol\\
&=& (2\underline{r}-2l+12)\Vol
\end{eqnarray*}
and the result follows.
\end{proof}
Now recall $\tau_{1,\mathrm{curv}}=-\frac{1}{3}\tilde{R}^\flat$ and recall the frame $\hat{\omega}_1,\hat{\omega}_2,\hat{\omega}_3$ of $\astil$-anti-selfdual 2-forms, cf. proposition \ref{tau1etau2}. We remark
\begin{equation}
 \phi=\alpha+\hat{\omega}_1e^4+\hat{\omega}_2e^5+\hat{\omega}_3e^6.
\end{equation}
Then the reader may check that the following result is correct, ie. $\tau_3'\in\Omega^3_{27}$. \begin{teo}\label{decompRalpha}
 $\calR\alpha=\frac{1}{7}(\underline{r}-l)*\phi+\frac{3}{4}\tau_{1,\mathrm{curv}}\phi+*\tau_3'$ with
\begin{equation}
 \begin{split}
  \tau_3'=-\tfrac{1}{7}(\underline{r}-l)\phi
+\bigl(R_{ij01}\astil e^{ij} -\tfrac{1}{4}(\tilde{R}_2\hat\omega_3-\tilde{R}_3\hat\omega_2)\bigr)e^4\\
+\bigl(R_{ij02}\astil e^{ij}
-\tfrac{1}{4}(\tilde{R}_3\hat\omega_1-\tilde{R}_1\hat\omega_3)\bigr)e^5\\
+\bigl(R_{ij03}\astil e^{ij}
-\tfrac{1}{4}(\tilde{R}_1\hat\omega_2-\tilde{R}_2\hat\omega_1)\bigr)e^6.
 \end{split}
\end{equation}
\end{teo}
The theorem says that $\calR\alpha$ acconts for the whole $\tau_{1,\mathrm{curv}}$ component.
Now an interesting phenomena occurs in a summand of (\ref{dphiedestrelaphi1}). One may patiently check the following:
\begin{equation}\label{mutbetamusigmatalpha2}
 \begin{split}
  (\mu T)\beta-\mu\sigma-T\alpha_2= 
(-Z_{132}e^{123}-Z_{023}e^{023}+Z_{310}e^{013}+Z_{201}e^{012})e^4\\
+(Z_{213}e^{123}+Z_{302}e^{023}+Z_{031}e^{013}-Z_{120}e^{012})e^5\\
-(Z_{321}e^{123}+Z_{230}e^{023}+Z_{103}e^{013}+Z_{012}e^{012})e^6.
 \end{split}
\end{equation}
The $Z_{...}$ were defined in (\ref{osZes}). The expression of this side of the torsion of $M$ on the \gwistor\ space preserves the orientation in some sense, cf. section \ref{EstudodaTorcao}. Thus the respective tensor $\tau_3$ can be deduced by the usual procedure: there is no component $\tau_0$, so we just subtract $\frac{3}{4}\tau_{1,\mathrm{tors}}\phi$, given in (\ref{tau1t}), and then take the Hodge dual. We leave this to the interested reader.

\subsection{The case of anti-$Z$ type torsion}
\label{TcoaZtt}

We have deduced that the condition $Z_{ijk}=0$, $\forall ijkl$ in direct order, is necessary and sufficient for the vanishing of $\tau_{1,\mathrm{tors}}$ and $\tau_{2,\mathrm{tors}}$, cf. proposition \ref{tau1tau2compodetorsao}. And from what we have just deduced in (\ref{mutbetamusigmatalpha2}) it is the same equivalent condition for the vanishing of $(\mu T)\beta-\mu\sigma-T\alpha_2$. Let us say a connection $D$ on $M$ is of the \textit{anti-$Z$ type} if $Z_{ijk}=T_{ijj}+T_{ikk}+T_{jkl}=0$, $\forall ijkl$ in direct ordering. This condition is studied in section \ref{EstudodaTorcao}.
\begin{prop}
 The following are equivalent:
\begin{meuenumerate}
\item $T^D$ is of the anti-$Z$ type.
\item $\tau_{1,\mathrm{tors}}=\tau_{2,\mathrm{tors}}=0$, this is $\dastT=0$.
\item $\beta\sigma+(\mu T)\alpha_1-\mu(T\alpha_1)=0$.
\item $(\mu T)\beta-\mu\sigma-T\alpha_2=0$.
\end{meuenumerate}
\end{prop}
As an important case to consider, we continue the study of \gwistor\ space tensors assuming $D$ to be of the anti-$Z$ type, cf. section \ref{EstudodaTorcao}. We then have
\begin{equation}\label{dphidestrelaphiemantiZtype}
\begin{split}
\dx\phi=\calR\alpha+(\underline{r}-l)\vol-\beta^2-2\mu\alpha_1\\
\dx*\phi=-\beta\varrho-\rho\vol.\hspace*{1.4cm}
\end{split}
\end{equation}
Summarizing the previous results in this setting we achieve:
\begin{teo}\label{astorsoesenfim}
Suppose $T^D$ is of the anti-$Z$ type. Then
\begin{equation}
\begin{split}
  \tau_0=\frac{2}{7}(\underline{r}-l+6)\quad\qquad\tau_1=-\frac{1}{3}\tilde{R}^\flat\qquad\quad 
\tau_2=\frac{1}{3}\tilde{R}\lrcorner(\phi-3\alpha)\\
\tau_3=\tau'_3+\tfrac{1}{7}(\underline{r}-l-2)(6\alpha-\mu\beta+\alpha_2).\qquad\qquad
\end{split}
\end{equation}
In particular, the \gwistor\ space is never a $G_2$ manifold, ie. its holonomy does not reduce to the exceptional group.
\end{teo}
\begin{proof}
 Indeed, comparing (\ref{dphidestrelaphiemantiZtype}) and proposition \ref{mualpha1beta2voldecompostos} and theorem \ref{decompRalpha}, we find $\tau_3$ given by
\begin{eqnarray*}
 & &\tau'_3-\tfrac{1}{7}(6\alpha-8\mu\beta-6\alpha_2)+ 
(\underline{r}-l)\tfrac{1}{7}(6\alpha-\mu\beta+\alpha_2)-\tfrac{1}{7}(6\alpha+6\mu\beta+8\alpha_2)\\
&=&\tau'_3+\tfrac{1}{7}(\underline{r}-l-2)(6\alpha-\mu\beta+\alpha_2).
\end{eqnarray*}
The last sentence follows, for instance, from the coefficient of $\alpha$ in $\tau_3$.
\end{proof}
Since $\underline{r}$ constant implies $r=\underline{r}g$, we may write another result.
\begin{teo}
Suppose $T^D=0$. Then $\tilde{R}_i=\ric U_i,\ i=1,2,3$, and
\begin{meuenumerate}
\item $\tau_0=0\Leftrightarrow\underline{r}=-6\Leftrightarrow \dx\phi=*\tau_3\Leftrightarrow (M,g)$ is Einstein with Einstein constant $-6$.
\item $\tau_1=0\Leftrightarrow\tau_2=0\Leftrightarrow (M,g)$ is Einstein.
\end{meuenumerate}
\end{teo}
From the point of view of $G_2$ geometry, another classification is due.
\begin{coro}
Suppose $T^D=0$. Then
\begin{meuenumerate}
\item $\tau_0=0\Leftrightarrow\underline{r}=-6\Leftrightarrow\dx\phi\wedge\phi=0\Leftrightarrow \dx\phi=*\tau_3\Leftrightarrow (SM,\phi)$ is of pure type $W_3$.
\item $\tau_1=0\Leftrightarrow\tau_2=0\Leftrightarrow(SM,\phi)$ is co-calibrated.
\end{meuenumerate}
\end{coro}
An example of a Riemannian manifold in the first case is the hyperbolic 4-space of constant sectional curvature $-2$. Still in the Levi-Civita case, suppose $R^D(X,Y)Z=c(\langle Y,Z\rangle X-\langle X,Z\rangle Y)$. Then it is easy to see $r=3cg$ and $\tilde{R}_i=0,\ \forall i$.
\begin{prop}
If $M$ has constant sectional Riemannian curvature $c$, then $\tau_1=\tau_2=0$ and
\begin{equation}
\tau_0=\frac{6}{7}(c+2),\qquad\ \ \tau_3=\frac{1}{7}\bigl((15c-12)\alpha+(2-6c)\mu\beta-(c+2)\alpha_2\bigr).
\end{equation}
\end{prop}
\begin{proof}
Since $R_{ijkl}=c(\delta_{il}\delta_{jk}-\delta_{ik}\delta_{jl})$ we get as part of $\tau'_3$
\[\sum_{k=1}^3\sum_{i<j}R_{ij0k}\astil e^{ij}e^{k+3}=-c\alpha_2.   \]
Thence $\tau'_3=-\frac{1}{7}3c\phi-c\alpha_2$. And $\tau_3$ follows, (satisfying $\tau_3*\phi=\tau_3\phi=0$).
\end{proof}
Of course it is not possible to have $\tau_3$ proportional to $\phi$. In the case of real $-2$-hyperbolic space $H^4$ we actually find an integrality:
\begin{equation}
\tau_0=0\qquad\quad \tau_3=-6\alpha+2\mu\beta.
\end{equation}
For the Hopf bundle $SS^4_{r_0}=SO(5)/SO(3)$ (standard embbeding) we have
\begin{equation}
 \dx\phi=6*\phi+*(9\alpha-4\mu\beta)
\end{equation}
where we chose $c=\frac{1}{r_0^2}=5$, the radius of the 4-sphere being $r_0$. This choice gave us an integrality again. It is interesting to notice that
\begin{equation}
 7|(c+2)\qquad\Rightarrow\qquad7|(15c-12)\quad\mbox{and}\quad7|(2-6c).
\end{equation}
The homogeneous space $S\C\Proj^2$ equals $SU(3)/U(1)$, clearly, and analogously with the hyperbolic Hermitian space $\C H^2$ (see below). The equations for $M$ a complex surface, in general, will be deduced in a separate article. For the moment we conjecture that the \gwistor\ space of $\C\Proj^2$ is not one of the well known nearly-parallel $G_2$ structures, cf. \cite{FriKaMoSe}, on Aloff-Wallach spaces. The case of $S^4$ already diverges from standard $S^7$ with nearly parallel structure. We do not know of any other $G_2$ structure on $S\C H^2=SU(1,2)/U(1)$ besides the present.

The real hyperbolic 4-space $H^4=SO_0(1,4)/SO(4)$, the sphere $S^4$, the hyperbolic Hermitian space $\C H^2=SU(1,2)/S(U(1)\times U(2))$ and $\C\Proj^2$ are the only irreducible Riemannian symmetric 4-spaces in which the action of the isotropy subgroup on $S_xM$, $\forall x\in M$, is transitive, cf. \cite{Besse}. This property characterizes rank 1 symmetric spaces (in any dimension). Those four are all Einstein.

According to \cite{Ishi} there are only six homogeneous Riemannian 4-manifolds: $H^4$, $S^4$, $\C\Proj^2$, the Euclidean 4-space seen in section \ref{ttgs}, $S^2\times S^2$ and $S^2\times\R^2$. But the classification of Einstein homogeneous 4-manifolds in \cite{Jen} gives more detail. They consist of the above four of rank 1, $\R^4$ flat, the product of two equal metric round spheres and the product of two equal metric real Poincar\'e disks $H^2$.

\section{Torsion of metric connections on 4-manifolds}
\label{EstudodaTorcao}

The study of the space of torsion tensors of metric connections was done by \'E. Cartan, cf. \cite{AgriThi}. Thence we know that under the orthogonal group $O(n,\R)$ there is the irreducible decomposition
\begin{equation}\label{decomptorsion}
 \Lambda^2\rn\otimes\rn=\rn\oplus{\cal A}\oplus \Lambda^3\rn
\end{equation}
where $\rn\oplus{\cal A}$ is the kernel of the Bianchi map $T_{ijk}\mapsto\bx T_{ijk}$ and ${\cal A}\oplus\Lambda^3$ is the kernel of the trace map $T_{ijk}\mapsto\sum_jT_{ijj}$. Clearly, ${\cal A}$ is the intersection of the two. Under the special orthogonal group the decomposition is the same as above, except in dimension 4. The so called metric connections with vectorial torsion, component $\R^n$, are characterized by their difference with the Levi-Civita connection being the tensor
\begin{equation}\label{vectorsion1}
 B_XY=\nu(X)Y-\langle X,Y\rangle \nu^\sharp
\end{equation}
with $\nu$ a fixed 1-form. Equivalently, $T^D=\nu\wedge1$.

Suppose we are in dimension $n=4$. Then $SO(4)=SO(3)\times_{\Z_2}SO(3)$ and
\begin{equation}
 {\cal A}={\cal A}_+\oplus{\cal A}_-
\end{equation}
is a decomposition induced by selfdual and anti-selfdual 2-forms. We have $\dim{\cal A}_+=\dim{\cal A}_-=8$. Notice $\Lambda^3\R^4=\R^4$ by duality so we have a second kind of vectorial torsion given by
\begin{equation}\label{vectorsion2}
 T^D={\cal X}\lrcorner\vol_M
\end{equation}
for some fixed vector field ${\cal X}$ on $M$. The coincidence of dimensions allows for the existence of invariant 4-subspaces inside $\Lambda^3\R^4\oplus\R^4$. In fact we shall describe the two subspaces ${\cal C}_+$ and ${\cal C}_-$ such that
\begin{equation}
 \Lambda^2_\pm\rn\otimes\rn={\cal C}_\pm\oplus{\cal A}_\pm.
\end{equation}

Suppose $D$ has torsion of the first vectorial kind: in an orthonormal basis, $\nu=\sum_{i=0}^3\nu_ie^i$ and  $T_{ijk}=\nu_\alpha(\delta_{\alpha i}\delta_{jk}-\delta_{\alpha j}\delta_{ik})$. Its trace is $3\nu$. One may compute the relevant tensors appearing in the torsion components of the \gwistor\ space structure studied in previous chapters, cf. proposition \ref{derivacoesexteriores}. We lift $\nu$ to a horizontal 1-form. Then it is easy to see, though we omit computations, that $\mu T=\nu\mu,\ T\alpha_1=\nu\alpha_1,\ T\alpha_2=2\nu\alpha_2$ and $\sigma=\beta\nu$. Then $\dx(\mu T)=(\dx\nu)\mu-\nu(-\beta+\mu T)=\mu\dx\nu+\sigma$. Since $\dx(\mu T)=\varrho+\sigma$, we conclude $\varrho=\mu\dx\nu$ and so $l=\varrho_{123}=0$. Moreover, those expressions appearing in $\dx\phi$ and $\dx*\phi$, cf. (\ref{dphiedestrelaphi1}) and (\ref{dphiedestrelaphi2}), which we saw to depend only on the torsion of $D$, verify
\begin{equation}\label{torsaovectorialtau1asterisophi}
\begin{split}
 -\beta\sigma-(\mu T)\alpha_1+\mu(T \alpha_1)\quad\qquad\qquad\\
=-\nu\beta^2-\nu\mu\alpha_1+\mu\nu\alpha_1=2\nu*\phi
\end{split}
\end{equation}
and so
\begin{equation}
\begin{split}
 (\mu T)\beta-\mu\sigma-T\alpha_2 =2\nu\mu\beta-2\nu\alpha_2=2\nu\phi-2\nu\alpha\quad\\
=\tfrac{3}{4}(2\nu)\phi+*\bigl(\tfrac{1}{2}\nu^\sharp\lrcorner(3\vol+\tfrac{1}{2}\beta^2+\mu\alpha_1)\bigr)
\end{split}
\end{equation}
It is worth knowing the formulae $*\nu\alpha=-\nu^\sharp\lrcorner\vol$, $*\mu\nu\beta=-\nu^\sharp\lrcorner\frac{1}{2}\beta^2$ and $*\nu\alpha_2=-\nu^\sharp\lrcorner\mu\alpha_1$. This was an example of a computation of $\tau_{3,\mathrm{tors}}$, one case of which we did not present in general. Finally, we also deduce the trace $m=3\mu(\nu^\sharp)$.

Now let us see the case of skew-symmetric vectorial torsion. Suppose ${\cal X}\lrcorner\vol= x_0e^{123}-x_1e^{023}+x_2e^{013}-x_3e^{012}$. Then a few computations let us find $T\alpha_1={\cal X}\lrcorner -\frac{1}{2}\mu\beta^2$, $T\alpha_2=-{\cal X}^\flat\mu\beta$, $\mu T=-{\cal X}\lrcorner\alpha_3$ and $\sigma={\cal X}\lrcorner\mu\alpha_2$. Clearly $m=0$. Thence $\dx(\mu T)=-{\cal L}_{\cal X}\alpha_3+{\cal X}\lrcorner\dx\alpha_3$ from which, since $\dx\alpha_3=\mu\alpha_2+m\vol$, follows $\varrho=-{\cal L}_{\cal X}\alpha_3$. Moreover, 
\begin{equation}\label{torsaoantisimetricatau1asteriscophi}
 -\beta\sigma-(\mu T)\alpha_1+\mu(T \alpha_1)={\cal X}^\flat*\phi
\end{equation}
and so
\begin{equation}
\begin{split}
 (\mu T)\beta-\mu\sigma-T\alpha_2 ={\cal X}^\flat*(-\alpha_2+\mu\beta)={\cal X}^\flat(\phi-\alpha)\\
=\tfrac{3}{4}{\cal X}^\flat\phi+*\bigl(\tfrac{1}{4}{\cal X}\lrcorner(3\vol+\tfrac{1}{2}\beta^2+\mu\alpha_1)\bigr).
\end{split}
\end{equation}
We observe that the expressions of the torsions for the two kinds of metric connections are quite similar (changing ${\cal X}^\flat$ for $2\nu$). The $G_2$ structure does not distinguish if there is classical vectorial torsion or a skew-symmetric torsion.

Now let us adress the question of anti-$Z$ type connection torsions, defined in section \ref{TcoaZtt}. The condition states
\begin{equation}
 Z_{ijk}=T_{ijj}+T_{ikk}+T_{jkl}=0
\end{equation}
for all $ijkl$ in direct order. Clearly, the space of solutions is 12 dimensional. It is the intersection of two larger spaces; that of $\tau_{1,\mathrm{tors}}=0$ and that of $\tau_{2,\mathrm{tors}}=0$. So let us solve these equations first.
Consider the first. It is given by the vanishing of $W_i$, $i=0,1,2,3$, where $W_i=Z_{ijk}+Z_{ilj}+Z_{ikl}$. If we look to equivalent equation (\ref{osWos2}), then we find a 20 dimensional solutions space. If we introduce a decomposition $T^D=T^1+T^2+T^3$ of a solution according to (\ref{decomptorsion}), by the respective order, then we find
\begin{equation}
(\ref{osWos2})\Leftrightarrow 2(3\nu_i)+3T^3_{jkl}=0\Leftrightarrow T^3_{jkl}=-2\nu_i.
\end{equation}
This is, $T^1=\nu\wedge1$ and $T^3=-2\nu^\sharp\lrcorner\vol_M$. This also confirms our previous computations in (\ref{torsaovectorialtau1asterisophi}) and (\ref{torsaoantisimetricatau1asteriscophi}). The space contains the 16 dimensional $\cal A$ part. Let
\[ {\cal C}_\pm=\bigl\{T(X,Y,Z)=\nu(X)\langle Y,Z\rangle-\nu(Y)\langle X,Z\rangle\:\pm\:2\nu^\sharp\lrcorner\vol_M(X,Y,Z) :\,\ \nu\in\R^n\bigr\} . \]
Then one checks that ${\cal C}_{\pm}\subset\Lambda_\pm^2\R^4\otimes\R^4$. Hence the solutions of $\tau_{1,\mathrm{tors}}=0$ are all the tensors in ${\cal A}\oplus{\cal C}_-$.

Now we solve equation $\tau_{2,\mathrm{tors}}=0$. According to (\ref{tau2t}), we have
\begin{equation}
  \frac{1}{3}W_i-Z_{ijk}=0,\ \ \mbox{for all}\ ijkl\ \mbox{in direct order}.
\end{equation}
Henceforth for each $i$ there is a linear system of rank 2
\begin{equation}
 \left\{\begin{array}{l}
 T_{jkk}-T_{jll}+T_{kil}-T_{ilk}=0\\
 T_{jkk}-T_{jii}+T_{lki}-T_{ilk}=0
\end{array}\right. .
\end{equation}
The space of solutions is hence 16 dimensional. Now, vectorial torsion and skew-symmetric torsion clearly satisfy the system, in particular confirming (\ref{torsaovectorialtau1asterisophi}) and (\ref{torsaoantisimetricatau1asteriscophi}): no $\tau_2$ component. So consider the selfdual and traceless torsion tensors $(e^{12}+e^{03})X_a+(e^{13}+e^{20})X_b+(e^{01}+e^{23})X_c$, with $X_a,X_b,X_c$ any vector fields. A simple example $T=(e^{12}+e^{03})e_2+(e^{01}+e^{23})e_0$ in ${\cal A}_+$ proves to be a solution, hence the solutions of $\tau_{2,\mathrm{tors}}=0$ are all the tensors in $\R^4\oplus{\cal A}_+\oplus\Lambda^3\R^4$.

Finally the anti-$Z$ type torsions, $\tau_{1,\mathrm{tors}}=\tau_{2,\mathrm{tors}}=0$ are those in ${\cal C}_-\oplus{\cal A}_+$.


\end{document}